\documentclass[12pt]{amsart}

\usepackage[colorlinks = true,
            linkcolor = red,
            urlcolor  = blue,
            citecolor = blue,
            anchorcolor = blue]{hyperref}

\numberwithin{equation}{section}

\newcommand{\ben}{\begin{enumerate}}
\newcommand{\een}{\end{enumerate}}

\newcommand{\bea}{\begin{eqnarray}}
\newcommand{\ba}{\begin{array}}
\newcommand{\bean}{\begin{eqnarray*}}
\newcommand{\ea}{\end{array}}
\newcommand{\eea}{\end{eqnarray}}
\newcommand{\eean}{\end{eqnarray*}}
\newcommand{\beq}{\begin{equation}}
\newcommand{\eeq}{\end{equation}}
\newcommand{\bthm}{\begin{thm}}
\newcommand{\ethm}{\end{thm}}
\newcommand{\blem}{\begin{lem}}
\newcommand{\elem}{\end{lem}}
\newcommand{\bprop}{\begin{prop}}
\newcommand{\eprop}{\end{prop}}
\newcommand{\bcor}{\begin{cor}}
\newcommand{\ecor}{\end{cor}}
\newcommand{\bdfn}{\begin{dfn}}
\newcommand{\edfn}{\end{dfn}}
\newcommand{\brem}{\begin{rem}}
\newcommand{\erem}{\end{rem}}
\newcommand{\bpf}{\begin{proof}}
\newcommand{\epf}{\end{proof}}
\newcommand{\bfact}{\begin{fact}}
\newcommand{\efact}{\end{fact}}
\newcommand{\bobs}{\begin{obs}}
\newcommand{\eobs}{\end{obs}}

\alph{enumii} \roman{enumiii}

\newtheorem{thm}{Theorem}[section]
\newtheorem{prop}[thm]{Proposition}
\newtheorem{lem}[thm]{Lemma}

\newtheorem{cor}[thm]{Corollary}
\newtheorem{dfn}[thm]{Definition}
\newtheorem{rem}[thm]{Remark}
\newtheorem{fact}[thm]{Fact}
\newtheorem{obs}[thm]{Observation}

                    \def\cC{\mathcal C}
             
\def\cL{{\mathcal L}}                   
                    
\def\cS{\mathcal S}
\def\cW{\mathcal W}

                      \def\R{{\mathbb R}}

\def\1{1\!\!1}
\def\and{\text{ and }}

        \def\diam{\text{\rm {diam}}}

      \def\lra{\longrightarrow}

\def\h{\rm{h}}
\def\hmu{\h_\mu}

\def\Int{\text{{\rm Int}}}
         \def\P{\text{{\rm P}}}

\def\L{{\mathcal L}}                   \def\Pa{{\mathcal P}}

\def\a{\alpha}                \def\b{\beta}             \def\d{\delta}
                         
                           \def\l{\lambda}
              \def\om{\omega}           
               \def\sg{\sigma}

\def\bi{\bigcap}              \def\bu{\bigcup}
\def\({\bigl(}                \def\){\bigr)}
\def\lt{\left}                \def\rt{\right}

\def\ld{\ldots}                        \def\^{\tilde}

\def\es{\emptyset}            \def\sms{\setminus}
\def\sbt{\subset}

      \def\imp{\Rightarrow}

\def\ov{\overline}

\def\om{\omega}

\begin{document}
%***********************************************
\title[]
{ \bf\large {Geometry of measures in random systems with complete connections}}
\date{\today}
% % author information %
\author[\sc Eugen MIHAILESCU]{\sc Eugen Mihailescu}
%\\[0.5cm]\rm \large Preliminary Version}
%\address
\address{Eugen Mihailescu,
Institute of Mathematics of the Romanian Academy,
Calea Grivitei 21, P.O Box 1-764,
RO 014700, Bucharest, Romania}
\email{Eugen.Mihailescu@imar.ro  \  \ \hspace*{0.5cm}
Web: www.imar.ro/$\sim$mihailes}
\author[\sc Mariusz URBA\'NSKI]{\sc Mariusz URBA\'NSKI}
\address{Mariusz Urba\'nski, Department of Mathematics,
 University of North Texas, Denton, TX 76203-1430, USA}
\email{urbanski@unt.edu \ \  \hspace*{0.5cm} Web: www.math.unt.edu/$\sim$urbanski}
%
% dedication %
%\dedicatory{}
%
% AMS information %
\date{}
\thanks{Research of the first author supported  by project PN-III-P4-ID-PCE-2020-2693 ``Dimensions and invariance in dynamical systems" from Ministry of Research and Innovation, CNCS/CCCDI-UEFISCDI Romania.  Research of the second author supported  by the NSF grant DMS 0400481.}

\begin{abstract}
 We study new relations between countable iterated function systems (IFS) with overlaps, Smale endomorphisms and random systems with complete connections. We prove that  stationary measures for countable conformal  IFS with overlaps and place-dependent probabilities, are exact dimensional; moreover we determine their Hausdorff dimension. Next, we construct a family of fractals in the limit set of a countable IFS with overlaps $\mathcal S$, and study the dimension for certain measures supported on these subfractals. In particular, we obtain  families of  measures on these subfractals which are related to the geometry of the system $\mathcal S$.
\end{abstract}
\maketitle


\section{Introduction}

In this paper we study relations between countable conformal iterated function systems (IFS) with arbitrary overlaps, Smale endomorphisms, and random systems with complete connections, from the point of view of their geometric and ergodic properties. We provide a common framework for studying measures with certain invariance properties and their dimensions in these systems.

Conformal iterated function systems were studied in many settings, and their invariant measures and  limit sets have attracted a lot of interest in the literature, for eg \cite{BDEG}-\cite{BV}, \cite{Fa}, \cite{FL}, \cite{FH}, \cite{H}, \cite{LNW}, \cite{M-CCM},  \cite{MU-Adv}, \cite{PS},  \cite{S}, to cite a few. Finite iterated function systems with place-dependent probabilities (weights) were introduced and studied by Barnsley, Demko, Elton, Geronimo in \cite{BDEG}; see also for eg \cite{BD}, \cite{BV} and others. The dimension theory for hyperbolic  endomorphisms was studied for eg in \cite{Pe}, \cite{Y},  \cite{M-stable}-\cite{MS}, and Smale endomorphisms were introduced and studied in \cite{MU-ETDS}.

Random systems with complete connections  were introduced and studied by Iosifescu and Grigorescu in \cite{IG} (see also \cite{GP}), and are generalizations of the chains with complete connections  introduced by Onicescu and Mihoc in \cite{OM2}. 

In the sequel, we first define/recall the notions of countable IFS with overlaps and place-dependent probabilities, the notion of random systems with complete connections, and the notion of Smale endomorphisms. A countable IFS with place-dependent probabilities is a particular case of random system with complete connections. Also one can associate random systems with complete connections,  to Smale endomorphisms.

In Section \ref{stationary}, given a countable IFS $\mathcal S$ with \textit{place-dependent probabilities} $\{p_i(\cdot), i \in I\}$ and \textit{arbitrary overlaps}, we find  its stationary measure.
Then  in Theorem \ref{infinitepdp} we prove the \textit{exact dimensionality} of such a stationary measure, and find its \textit{Hausdorff dimension} (and pointwise, box dimension). In general, the exact dimensionality largely characterizes the local and global metric properties of the respective measure.

Next, in Section \ref{maximal} for an arbitrary countable IFS with overlaps $\mathcal S$ which satisfies a condition of pointwise non-accumulation, we associate a \textit{maximal Smale endomorphism} to it. Using this method we  form a family of \textit{random subfractals} in the limit set $\Lambda$ of $\mathcal S$, corresponding to subsystems of iterations. In Theorem \ref{exactmax} we determine the \textit{pointwise dimension} for a class of invariant measures supported on these random subfractals in $\Lambda$. In particular we obtain families of  measures associated to certain potentials $\psi_s$, which are related to the geometry and overlappings of the system $\mathcal S$.

\

Recall first that a finite measure $\mu$ on a metric space $X$ is  \textit{exact dimensional} (for eg \cite{Pe}, \cite{Y}) if  there exists a number $ \delta \ge 0$ such that for $\mu$-a.e $x \in X$, $$\mathop{\lim}\limits_{r \to 0}\frac{\log \mu(B(x, r))}{\log r} = \delta.$$ In this case, it follows that $HD(\mu) = \delta$. Exact dimensionality of a measure $\mu$ is a strong geometric property, and implies that the fractal dimensions of $\mu$ (Hausdorff, pointwise, box dimension) coincide. 

For finite conformal IFS with overlaps, the exact dimensionality of projections of ergodic  measures was proved in \cite{FH}. For countable conformal IFS with overlaps, the exact dimensionality of projections of ergodic measures satisfying a finite entropy condition was proved in \cite{MU-Adv}. This property  was also studied  for hyperbolic diffeomorphisms for eg in \cite{Pe}, \cite{Y}, and for hyperbolic endomorphisms (non-invertible maps) for eg in \cite{M-stable}. 

\

Now, let us recall the three main notions  used in the sequel:

The notion of \textit{finite iterated function systems with place-dependent probabilities} was introduced in \cite{BDEG}; see also \cite{BD}, \cite{BV}, \cite{FL}, \cite{L}, \cite{LP}, \cite{S}. These iterated function systems are particular cases of chains with complete connections introduced in \cite{OM2}; see also the papers \cite{DF}, \cite{ITM}. 

A \textit{chain with complete connections} is a sequence of random variables $\xi_1, \xi_2, \ldots$ taking real values in a countable set $\Omega$, where the probability that at step $n$, $\xi_n$ takes value $\omega \in \Omega$, depends on the values taken by all previous random variables. Thus,
$$
P(\xi_{n+1} = \omega|\xi_n = \omega_0, \xi_{n-1}=\omega_{-1}, \ldots) = P(c, \omega),
$$
where $\omega \in \Omega$ and $c = (\ldots, \omega_{-1}, \omega_0)$ is a trajectory in 
$
\Sigma^-_\Omega:= \mathop{\prod}\limits_{j\in \mathbb Z, j \le 0} \Omega
$.
  One assumes in general that $P^1(c, \omega) = P(c, \omega)$ and for every $n \ge 1$, 
$$
P^{n+1}(c, \omega) = \mathop{\sum}\limits_{\omega' \in \Omega} P(c, \omega') P^n(a(c, \omega'), \omega),
$$
where  $a(c, \omega'):= (\ldots, \omega_{-1}, \omega_0, \omega')$, for $c = (\ldots, \omega_{-1}, \omega_0)\in \Sigma^-_\Omega$.

In the sequel, we extend  the notion of finite IFS with place-dependent probabilities to \textit{countable iterated function systems with overlaps and place-dependent probabilities}. Such a system $\mathcal S$ is determined by the continuous functions   
$
\phi_i: V \lra V, \  \  i \in I
$, defined on a compact set $V \subset \mathbb R^D$ 
and indexed by a countable set $I$, and by the continuous probability functions (weights),  
$
p_i: V \lra [0, 1], \  i \in I
$,
satisfying
$$
\mathop{\sum}\limits_{i\in I} p_i(x) = 1. $$

 By \textit{IFS with overlaps} we mean that the sets $\phi_i(V), i \in I$ may intersect in any way.  So the overlaps are arbitrary, and we do not assume any kind of Open Set Condition (\cite{H}, \cite{Fa}). 
 Assume also that there exists a number $s \in (0, 1)$ such that,
\begin{equation}\label{infIFS}
|\phi_i(x) - \phi_i(y) | \le s |x-y|, \ \forall i \in I, \ x, y \in V.
\end{equation}

For properties of (finite or countable) conformal iterated systems with various  types of overlaps see for eg \cite{Fa}, \cite{FH}, \cite{LNW}, \cite{M-CCM}, \cite{MU-CM}, \cite{NT}, \cite{PS}. 
The countable IFS case is different from the finite case, since the fractal limit set may be \textbf{non-compact}, and many methods from the finite case do not work. 

\

For $x \in V$ and a Borel set $B \subset V$, the \textit{probability of transfer} from $x$ to $B$ is equal to
$$
P(x, B) = \mathop{\sum}\limits_{i \in I} p_i(x) \delta_{\phi_i(x)}(B)
.$$
We have then the associated \textit{transfer operator}: 
$$
\L g(x) = \int_V g(y) P(x, dy) = \mathop{\sum}\limits_{i
\in I} p_i(x) g(\phi_i(x)),
$$
where $g:V \to \mathbb R$ is measurable. If $M(V)$ denotes the space of finite signed Borel measures on $V$, then the operator $\L^*$ adjoint to $\L$, restricted to the space $M(V)$, is given by  
\begin{equation}\label{Vmu}
\L^*\mu(B) = \int P(x, B) d\mu(x) = \mathop{\sum}\limits_{i \in I} \int_{\phi_i^{-1}B} p_i(x) d\mu(x)
\end{equation}
A Borel probability measure $\mu$ on $V$ is called   \textit{stationary} for the above system if 
\begin{equation}\label{L}
\L^*\mu = \mu,
\end{equation}
and \textit{attractive} if for all probabilities $\nu$ on $V$ and all bounded measurable $g:V \to \R$, 
$$
\lim_{n\to\infty}\int_X g\, d(\L^{*n} \nu) = \int_V g\, d\mu.
$$

One of the central problems in the theory of chains with complete connections and of IFS with place--dependent probabilities is to find stationary measures and to study their ergodic and metric properties. 
In many of these results in the finite case, the probability functions $p_i(\cdot)$  satisfy a H\"older type condition (\cite{BDEG}, \cite{BD}, \cite{IG}).

\

The second main notion is that of \textit{random systems with complete connections},  which are generalizations of chains with complete connections (for eg \cite{GP}, \cite{IG}). 
\begin{dfn}[\cite{IG}]\label{RSCC}
A random system with complete connections (or RSCC) is a quadruple $((W, \mathcal W), (X, \mathcal X), u, P)$ where:

i) $(W, \mathcal W)$ and $(X, \mathcal X)$ are measurable spaces, 

ii) $u: W\times X \to W$ is a measurable map, with the product $\sigma$-algebra $\mathcal W \times \mathcal X$ on $W\times X$, 

iii) $P$ is a transition probability function from $(W, \mathcal W)$ to $(X, \mathcal X)$, i.e. $P(w, \cdot)$ is a probability on $\mathcal X$ for any $w \in W$ and $P(\cdot, A)$ is a random variable on $Z$ for any set $A \in \mathcal X$.

We call $W$ the \textit{state space} and $X$ the \textit{index space}.  $W$ is assumed to be a locally compact and $\sigma$-compact metric space. For $\mathcal W$ we take the $\sigma$-algebra generated by open sets of $W$.
\end{dfn}

 An example of RSCC is  an urn scheme with replacement. Consider an initial urn $U_0$, which  contains $a_j = a_j^{(0)}$ balls of color $j,  1\le j \le m$.  If on trial $n\ge 1$ we extract a ball of color $j$, then this ball is replaced together with $d_j$ balls of same color (the rest of the balls being left unchanged), hence $a_i^{(n)} = a_i^{(n-1)} + \delta_{ij} d_i, 1 \le i \le m,$
 with $d_1, \ldots, d_m$ being non-negative integers. Thus the probability of choosing a certain color at step $n$ depends on all previous steps.

 \begin{rem}\label{ifsrsc}
If in Definition \ref{RSCC} the index space $X$ is countable and $\mathcal X$ is the algebra of subsets of $X$, and we define the maps $\phi_i(w):= u(w, i)$, then we obtain a countable IFS on $W$. Denote $x^{(n)}:= (x_1, \ldots, x_n) \in X^n$, $n \ge 1$. By induction define  $u^{(n)}:W \times X^n \to W$,
$$u^{(n+1)}(w, x^{(n+1)}) = u(w, x_1), n = 0, \ \text{and} \ u^{(n+1)}(w, x^{(n+1)}) = u(u^{(n)}(w, x^{(n)}), x_{n+1}), n \ge 1$$ 
We can write $wx$ for $u(w, x)$, and $wx^n$ for $u^{(n)}(w, x^{(n)})$.
For $w \in W$ and $A$ Borel set in $X$, let $P_1(w, A) = P(w, A)$, and if  $w \in W, m > 1$ and $A$ Borel set in $X^m$, let the $m$-transfer probability of $w$ into $A$ be,
$$P_m(w, A) = \int_X P(w, dx_1) \int_X P(wx_1, dx_2) \ldots \int_XP(wx^{(m-1)}, dx_m) d\chi_A(x^{(m)}$$
And for $w \in W, n, m \ge 1$ and $A\subset X^m$, define \ $P_m^n(w, A) = P_{n+m-1}(w, X^{n-1}\times A)$.
\end{rem}

The following existence result was proved in \cite{IG}.
\begin{thm}\label{existsta}
Let a random system with complete connections $\{(W, \mathcal W), (X, \mathcal X), u, P)$ and an arbitrary given point $w_0 \in W$. 
 Then, there exist a probability space $(\Omega, \mathcal K, P_{w_0})$ and a sequence $(\xi_n)_{n \ge 1}$ of $X$-valued random variables defined on $\Omega$ such that, for all $m, n, q \ge 1$ and $A\in \mathcal X^m$, we have:

i) $P_{w_0}([\xi_n, \ldots, \xi_{n+m-1}]\in A) = P_m^n(w_0, A)$, 

ii) $P_{w_0}([\xi_{n+q}, \ldots, \xi_{n+q+m-1}]\in A|\xi^{(n)}) = P_m^q(w_0\xi^{(n)}, A)$, $P_{w_0}$-a.e.

\end{thm}

\

Finally, the third main notion we will use is that of \textit{Smale skew product endomorphism}, introduced in \cite{MU-ETDS}. 
Let $I$ be a countable alphabet, and let $\Sigma_I^+$ be the associated 1-sided shift space, with shift map $\sigma: \Sigma_I^+ \to \Sigma_I^+$. 
Given $\b>0$, the metric $d_\b$ on $\Sigma_I^+$ is:
$$
d_\b\((\om_n)_0^{\infty},(\tau_n)_0^{\infty}\)
=\exp\(-\b\max\{n\ge 0:(0\le k\le n) \imp \om_k=\tau_k\}\)
$$
with the standard convention that $e^{-\infty}=0$. All metrics $d_\b$, $\b>0$, on $\Sigma_I^+$ are H\"older continuously equivalent
and all induce the product topology on $\Sigma_I^+$. 
We define also the 2-sided shift space $\Sigma_I$ with the same metric as above.

For every $\om\in \Sigma_I$ and integers $m\le n$, define the $(m, n)$-truncation
$
\om|_m^n=\om_m\om_{m+1}\ld\om_n.
$
Let $\Sigma_I^*$ be the set of finite words. For
 $\tau=\tau_m\tau_{m+1}\ld\tau_n$, the cylinder from $m$ to $n$ is
$
[\tau]_m^n=\{\om\in \Sigma_I: \om|_m^n=\tau\}
$.
The family of cylinders from $m$ to $n$ is
denoted by $C_m^n$. 
If $m=0$, write $[\tau]$ for $[\tau]_m^n$. 

Let $\psi: \Sigma_I \to\R$ continuous. Topological pressure plays an important role in thermodynamic formalism and extends the notion of entropy (for eg \cite{Ru}). By extension, in our case the \textit{topological pressure} of $\psi$ is:
$\P(\psi)$ is,
\beq\label{1_2015_11_04}
\P(\psi)
:=\lim_{n\to\infty}\frac1n\log\sum_{\om\in C_0^{n-1}}\exp\(\sup\(S_n\psi|_{[\om]}\)\),
\eeq
where the limit above exists by subadditivity. 
 A shift-invariant
Borel probability $\mu$ on the 2-sided shift space $\Sigma_I$ with countable alphabet $I$,  is called a \textit{Gibbs measure}
of $\psi$ if there exist constants $C\ge 1$, $P\in\R$
such that
\beq\label{1082305}
C^{-1}
\le {\mu([\om|_0^{n-1}])\over \exp(S_n\psi(\om)-Pn)}
\le C
\eeq
for all $n\ge 1, \om\in \Sigma_I$.  From (\ref{1082305}), if $\psi$ admits a Gibbs
state, then $P=\P(\psi)$. 
The function $\psi:\Sigma_I \to\R$ is called \textit{summable} if 
$$
\sum_{e\in E}\exp\(\sup\(\psi|_{[e]}\)\)<\infty
$$
In \cite{MU-ETDS} we  proved that a H\"older continuous $\psi: \Sigma_I \to\R$ is summable if and only if $\P(\psi)<\infty$, and that for every H\"older continuous summable function $\psi: \Sigma_I \to\R$ 
there exists a unique Gibbs state $\mu_\psi$ on $\Sigma_I$, and the measure
$\mu_\psi$ is ergodic. We also showed that if $\psi: \Sigma_I \to\R$ is a H\"older continuous summable 
function,  then
$$
\sup\lt\{\hmu(\sg)+\int_{\Sigma_I}\!\!\!\psi d\mu:\mu\circ\sg^{-1}=\mu 
   \  \text{ and }  \  \int\!\!\psi d\mu>-\infty\rt\}
=\P(\psi)
={\hmu}_\psi(\sg)+\int_{\Sigma_I}\!\!\psi d\mu_\psi,
$$
and this supremum is reached only at $\mu_\psi$. We mention that similar results were proved in \cite{gdms} on the 1-sided shift space $\Sigma_I^+$, and in \cite{MU-ETDS} we extended them  on the 2-sided space $\Sigma_I$.

Consider now the measurable partition of $\Sigma_I$,
\begin{equation}\label{parti}
\Pa_-=\{[\om|_0^{\infty}]:\om\in \Sigma_I\}
     =\{[\om]:\om\in \Sigma_I^+\}.
\end{equation}

 If $\mu$
is a probability on $\Sigma_I$, let the family of canonical conditional measures of $\mu$ associated to $\Pa_-$ (\cite{Ro}), $$
\{\ov\mu^\tau:\tau\in \Sigma_I^+\}.
$$ 
Then $\ov\mu^\tau$ is a probability
measure on the cylinder $[\tau|_0^{+\infty}]$ and we also write $\ov\mu^\om$, $\om\in \Sigma_I^+$, for the 
conditional measure on $[\om]$. The canonical projection (truncation) is:
$$
\pi_0: \Sigma_I \to \Sigma_I^+, \  \pi_0(\tau) = \tau|_0^\infty, \tau \in \Sigma_I,
$$
Recall that the system $\{\ov\mu^\om:\om\in \Sigma_I^+\}$ of conditional measures is uniquely determined (up to measure zero), by the property (\cite{Ro}),
$$
\int_{\Sigma_I}g\,d\mu=\int_{\Sigma_I^+}\int_{[\om]}g\,d\ov\mu^{\om}
  \,d(\mu\circ \pi_0^{-1})(\om), \
\forall g\in L^1(\mu).$$ 

\

We introduced the notion of \textit{Smale skew product endomorphisms}.

\begin{dfn}\label{Smaleskew}\cite{MU-ETDS}
Let $(Y,d)$ be a complete bounded metric space. For every $\om
\in \Sigma_I^+$ let $Y_\om\sbt Y$ be an arbitrary set and let
$
T_\om:Y_\om\lra Y_{\sg(\om)}
$
be a continuous injective map. Let  
$
\hat Y:=\bu_{\om\in \Sigma_I^+}\{\om\}\times Y_\om\sbt \Sigma_I^+\times Y,
$
and define the map 
$
T:\hat Y\lra\hat Y, 
T(\om,y)=(\sg(\om),T_\om(y)).
$
The pair $(\hat Y,T:\hat Y\to\hat Y)$ is called a \textit{model
Smale endomorphism} if there exists $\l>1$ such 
that for all $\om\in \Sigma_I^+$ and all $y_1,y_2\in Y_\om$,
$d(T_\om(y_2),T_\om(y_1))\le \l^{-1}d(y_2,y_1)$.
\end{dfn}

If $\tau = (\tau_{-n}, \ldots, \tau_0, \tau_1, \ldots)$ let  
$
T_\tau^n
=T_{\tau|_{-1}^{\infty}}\circ T_{\tau|_{-2}^{\infty}}\circ\ld
   \circ T_{\tau|_{-n}^{\infty}}:Y_\tau\lra Y_{\tau|_0^{\infty}}
$. If $\tau\in \Sigma_I$ let, 
$$
T_\tau^n
:=T_{\tau|_{-n}^{\infty}}^n
:=T_{\tau|_{-1}^{\infty}}\circ T_{\tau|_{-2}^{\infty}}\circ\ld
  \circ T_{\tau|_{-n}^{\infty}}:Y_{\tau|_{-n}^{\infty}}\longrightarrow Y_{\tau|_0^{\infty}}
$$
Then  the sets $\(T_\tau^n\(Y_{\tau|_{-n}^{\infty}}\)\)_{n=0}
^\infty$ form a  descending sequence, and 
$
\diam\(\ov{T_\tau^n\(Y_{\tau|_{-n}^{\infty}}\)}\)\le \l^{-n}\diam(Y).
$
As $(Y,d)$ is complete, $
\bi_{n=1}^\infty \ov{T_\tau^n\(Y_{\tau|_{-n}^{\infty}}\)}
$
is a point denoted by $\hat\pi_2(\tau)$, which defines the map
\begin{equation}\label{pi2hat}
\hat\pi_2: \Sigma_I \lra Y,
\end{equation}
and  define also $\hat\pi: \Sigma_I \to \Sigma_I^+\times Y$ by 
\beq\label{5111705p141}
\hat\pi(\tau)=\(\tau|_0^{\infty},\hat\pi_2(\tau)\),
\eeq
and the truncation to non-negative indices by 
$\pi_0: \Sigma_I \lra \Sigma_I^+, \  \  \pi_0(\tau) = \tau|_0^\infty$.

Now assume $Y_\omega = Y, \forall \omega \in \Sigma_I^+$, and for an arbitrary $\om\in \Sigma_I^+$ denote the $\hat \pi_2$-projection of the cylinder $[\om]\subset \Sigma_I$, 
$
J_\om:=\hat\pi_2([\om])\subset Y,
$
and call these sets the \textit{stable Smale fibers} of  $T$. The global invariant set,
$$
J:=\hat\pi(\Sigma_I)=\bu_{\om\in \Sigma_I^+}\{\om\}\times J_\om\sbt \Sigma_I^+\times Y,
$$
is called the \textit{Smale space} induced by $T$. 
 Then the skew-product system
\begin{equation}\label{J}
T:J\lra J,
\end{equation} is called the \textit{Smale endomorphism} generated by  $T:\hat Y\lra \hat Y$.

Now suppose more conditions about  $Y_\om$, $\om\in \Sigma_I^+$ and  the maps $T_\om:Y_\om\to Y_{\sg(\om)}$, namely:
\begin{itemize}
\item[(a)] $Y_\om$ is a closed bounded subset of $\R^d$, with some $d\ge 1$ such that $\ov{\Int(Y_\om)}=Y_\om$. 

\item[(b)] Each map $T_\om:Y_\om\to Y_{\sg(\om)}$ extends to a $C^1$ conformal embedding from $Y_\om^*$ to $Y_{\sg(\om)}^*$, where $Y_\om^*$ is a bounded connected open subset of $\R^d$ containing $Y_\om$. Then $T_\om$ denotes also this extension and assume that the maps
$
T_\om:Y_\om^*\to Y_{\sg(\om)}^*$ satisfy:
\item[(c)]  There is $\lambda>1$ so that 
$d(T_\om(y_1), T_\om(y_2))\le \l^{-1}d(y_1,y_2), \forall \om\in \Sigma_I^+,  y_1,y_2\in Y_\om^*$.
\item[(d)] (Bounded Distortion Property 1) There are constants $\a>0, H>0$ s.t $\forall y, z\in Y_\om^*$, \ 
$$
\big|\log|T_\om'(y)|-\log|T_\om'(z)|\big|\le H||y-z||^\a.
$$
\item[(e)] The function 
$
\Sigma_I\ni\tau\longmapsto\log|T_\om(\hat\pi_2(\eta))|\in\R
$
is H\"older continuous, where $\om = \pi_0(\tau)$.
\item[(f)] (Open Set Condition) For every $\om\in \Sigma_I^+$ and for all $a, b\in I$ with  $a\ne b$, we have \ 
$
T_{a\om}(\Int(Y_{a\om}))\cap T_{b\om}(\Int(Y_{b\om}))=\es.
$
\item[(g)] (Strong Open Set Condition) There exists a measurable function $\d:\Sigma_I^+\to(0,\infty)$ such that for every $\om \in \Sigma_I^+$, \ 
$
J_\om\cap\(Y_\om\sms\ov B(Y_\om^c,\d(\om)\)\ne\es.
$
\end{itemize}

 A skew product Smale endomorphism satisfying conditions (a)--(g) will be called in the sequel a \textit{conformal Smale endomorphism}. 

\

We see that a countable IFS with place-dependent probabilities is a particular case of random system with complete connections with state space $V$,  index space =  countable alphabet $I$, and probability transition function determined by  $p_i(\cdot), i \in I$, namely
\begin{equation}\label{ptf}
P(x, B) = \mathop{\sum}\limits_{i \in I} p_i(x) \delta_{\phi_i(x)}(B)
\end{equation}

\

Also to a conformal Smale endomorphism $T$ with fibers $J_\om, \om \in \Sigma_I^+$, we can associate a random system with complete connections. Recall that $\pi_0:\Sigma_I \to \Sigma_I^+, \pi_0(\eta) = \eta|_0^\infty, \eta \in \Sigma_I$.

\begin{thm}\label{ev}
Let  the conformal Smale endomorphism $T: J \to J$ defined in (\ref{J}), where we assume the spaces  $Y_\omega = Y\subset \R^d, \forall \om \in \Sigma_I^+$. Consider also a shift  invariant measure $\mu$ on $\Sigma_I$. Define the quadruple $((W, \mathcal W), (X, \mathcal X), u, P)$ by:

 a) State space $W = \{(\om, x), \ x \in J_\om, \om \in \Sigma_I^+\}$ with the Borel $\sigma$-algebra $\cW$ induced from the Borel $\sigma$-algebra of the product space $\Sigma_I^+ \times Y\subset \Sigma_I^+ \times \R^d$, and index space $X = \Sigma_I^+$ with its Borel $\sigma$-algebra $\mathcal X$;
 
 b) For $\tau \in \Sigma_I^+$ and $w \in W, w = (\om, x), x \in J_\om$, define the map $u(w, \tau) = (\sigma \om, T_\om(x))$;
  
  c) For those $\om \in \Sigma_I^+$ for which the conditional measure $\bar \mu^\om$ is defined (their set has $\pi_{0*}\mu$-measure equal to $1$), and for $w = (\om, x), x \in J_\om$, define the probability transition function $P_w(\cdot)$ by $ P_w := \bar\mu^\om$.

  Then, $((W, \mathcal W), (X, \mathcal X), u, P)$ is a random system with complete connections.
\end{thm}

\begin{proof}
On $W$ we take the $\sigma$-algebra of Borel sets induced from $\Sigma_I^+ \times \R^d$.
We defined the map $u(w, \tau) = (\sigma\om, T_\om(x))$, for $w = (\om, x) \in W$, $\om \in \Sigma_I^+$ and $x \in J_\om$. On the other hand recall that $J_\om$ is the set of points of the form $\hat\pi_2(\eta)$ for $\eta \in [\om] \subset  \Sigma_I$. So from (\ref{pi2hat}) the map $u(\cdot, \cdot)$ is well-defined, since if $x = \hat\pi_2(\eta) = T_{\eta_{-1}\om}\circ T_{\eta_{-2}\eta_{-1}\om} \circ \ldots \in J_\om$, then $$T_\om(x) = \hat\pi_2(\sigma\eta) \in J_{\sigma\om},$$ as $\sigma\eta \in [\sigma\om]$.
Then, we use condition (e) from the definition of conformal Smale endomorphisms, to obtain that $u: W\times X \to W $ is measurable. 

 Next, for $\pi_{0*}\mu$-a.e $\om \in \Sigma_I^+$, the conditional measure $\bar \mu^\om$ is defined on the cylinder $[\om]\subset \Sigma_I$, and this cylinder  can be identified with $\Sigma_I^+$.
If $A$ is a Borel set in $\Sigma_I^+$ and  $w = (\om, x) \in W, x \in J_\om$,  define $P(w, A) = \bar\mu^\om(A)$, thus $P_w$ can be viewed as a probability measure on $\Sigma_I^+$. 
From the uniqueness of the system of conditional measures associated to the partition $\mathcal P_-$ (see \cite{Ro}) with the property that $$\int_{\Sigma_I} g(\xi) d\mu = \int_{\Sigma_I^+}\int_{[\om]} g(\xi) d\bar\mu^\om(\xi) d\pi_{0*}\mu(\om),$$ for any integrable function $g: \Sigma_I \to \R$, we obtain that for any set $A$  as above, $P_w(A)$ depends measurably on $\om \in \Sigma_I^+$, where $w = (\om, x), x \in J_\om$.  
Hence the function $$P(\cdot, A): W \to \R, \ w \to P(w, A)$$ is measurable with respect to the $\sigma$-algebra $\mathcal W$ induced on $W$ from $\Sigma_I^+ \times \R^d.$
\end{proof}

\section{Stationary measures for countable systems with overlaps and place-dependent probabilities}\label{stationary}
 
In this section we study the case of \textit{countable IFS with overlaps and place-dependent probabilities}, and prove the exact dimensionality of stationary measures, and compute the pointwise (and Hausdorff) dimension. 

Consider a system of smooth contractions defined on a compact set $V \subset \mathbb R^D$  indexed by a countable alphabet $I$, 
$
\cS=\big\{\phi_i:V \lra V \big\}_{i \in I}
$
with limit set $\Lambda$, and the weights $p_i: V \to\R$, $i\in I$, satisfying for any $x \in V$,
\beq\label{sum1}
\mathop{\sum}\limits_{i \in I} p_i(x) = 1.
\eeq

 $\Sigma_I^+$ denotes the 1-sided shift space with alphabet $I$. If $i_1, \ldots, i_n \in I, n \ge 1$,  denote $$\phi_{i_1\ldots i_n} := \phi_{i_1}\circ \ldots \circ \phi_{i_n}.$$
Let $\pi:\Sigma_I^+ \to \Lambda, \pi(\om) = \mathop{\lim}\limits_{n\to \infty}\phi_{\om_0 \om_1 \ldots \om_n}$ if $\om=(\om_0, \om_1, \ldots) \in \Sigma_I^+$, be the canonical coding map for the limit set $\Lambda$.
 
Assume also that $p_i(\cdot)$ depend uniformly H\"older continuously on $x\in V$, for $i \in I$, i.e. there exist constants $\alpha, C >0$ such that for all $i \in I$ and all $x, y \in V$,
\begin{equation}\label{Ho}
|p_i(x) - p_i(y)| \le C |x-y|^\alpha.
\end{equation}
The \textit{transfer probability} in this case is 
$P(x, B):= \mathop{\sum}\limits_{i \in I} p_i(x)\delta_{\phi_i(x)}(B)$,
and the \textit{transfer operator} $\cL: \mathcal C(V) \to \cC(V)$ is given by (see for eg \cite{BDEG}):
$$\cL(f)(x) = \int_X f(y) P(x, dy).$$
A measure $\mu$ on $V$ is called \textit{stationary} if it is a fixed point of the dual operator of $\cL$,
$$\cL^*(\nu)(B) = \int P(x, B) d\nu(x) = \mathop{\sum}\limits_{i \in I} \int_{\phi_i^{-1}(B)} p_i(x) d\nu(x).$$

Define also the \textit{Lyapunov exponent} of a shift-invariant measure $\mu$ on $\Sigma_I^+$ by:
$$\chi_\mu := -\int_{\Sigma_I^+} \log |\phi'_{\omega_0}(\pi(\sigma\omega))| \ d\mu(\omega).$$

Let us now recall the notion of \textit{projection entropy} for an invariant measure of a countable iterated function system from \cite{MU-Adv}. This is a generalization of the notion of projection entropy from the finite case of \cite{FH}, and is useful in formulas for the dimension of invariant measures in IFS with overlaps (see \cite{FH}, \cite{M-CCM}, \cite{MU-Adv}). 

In \cite{MU-Adv} we defined the projection entropy in the more general case of random countable iterated function systems, but here we need it only for countable deterministic systems. 
So let $\mathcal S = \{\phi_i, i \in I\}$ be a countable system and $\mu$ be a $\sigma$-invariant probability measure on $\Sigma_I^+$. Denote by $\xi$ the partition of $\Sigma_I^+$ into initial 1-cylinders, and by $\epsilon_{\R^D}$ the point partition of $\R^D$, and by $\pi:\Sigma_I^+ \to \Lambda$ the canonical coding map for the limit set $\Lambda\subset \R^D$ of the function system $\mathcal S$. Then $\pi^{-1}\epsilon_{\R^D}$ and $\sigma^{-1}(\pi^{-1}\epsilon_{\R^D})$ are measurable partitions of $\Sigma_I^+$.  The \textit{projection entropy} of $\mu$ with respect to $\mathcal S$ is defined then by, 
\begin{equation}\label{pent}
h_\mu(\mathcal S) := H_\mu(\xi|\sigma^{-1}(\pi^{-1}\epsilon_{\R^D})) - H_\mu(\xi|\pi^{-1}\epsilon_{\R^D}).
\end{equation}

\

We prove now that the stationary measure from Theorem \ref{existsta} is exact dimensional.

\begin{thm}\label{infinitepdp}
In the above setting if the system $\mathcal S =\{\phi_i, i \in I\}$ is countable and conformal and if the probabilities $\{p_i(\cdot), i \in I\}$ satisfy (\ref{sum1})-(\ref{Ho}), then the stationary measure $\tilde\mu_P$ for the system $\mathcal S$ with place-dependent probabilities $P = \{p_i(\cdot), i \in I\}$ is exact dimensional, and $$HD(\tilde\mu_P)  =  \frac{h_{\mu_\psi}(\mathcal S)}{\chi_{\mu_\psi}}\le \frac{h_{\mu_\psi}(\sigma)}{\chi_{\mu_\psi}},$$ where $h_{\mu_\psi}(\mathcal S)$ is the projection entropy of $\mu_\psi$ with respect to $\mathcal S$, and $\psi: \Sigma_I^+ \to \mathbb R,  \psi(\omega):= \log p_{\omega_0}(\pi(\sigma \omega)),  \omega \in \Sigma_I^+$, and $\mu_\psi$ is the equilibrium measure of $\psi$ on $\Sigma_I^+$.

\end{thm}

\begin{proof}
First let us define  the potential $\psi: \Sigma_I^+ \to \mathbb R$ by:
$$
\psi(\omega):= \log p_{\omega_0}(\pi(\sigma \omega)), \ \omega \in \Sigma_I^+.$$
 From the conditions (\ref{sum1}) and (\ref{Ho}), it follows  that $\psi$ is summable and H\"older continuous on $\Sigma_I^+$. Then there exists an equilibrium measure $\mu_\psi$ on $\Sigma_I^+$, which projects to the probability measure $\nu_p$ on $\Lambda$, associated to the system of weights $\textbf p:=(p_i, i \in I)$. 
The transfer operator $\cL: \mathcal C(\Sigma_I^+) \to C(\Sigma_I^+)$ is in this case, 
$$\cL(\phi)(\omega):= \mathop{\sum}\limits_{i=1}^\infty p_i(\pi\omega) \phi(i\omega),$$
where $\mathcal C(\Sigma_I^+)$ is the space of (bounded) continuous real-valued functions on $\Sigma_I^+$.
We see from above that $$\cL(\phi)(\omega) = \mathop{\sum}\limits_{i=1}^\infty e^{\psi(i\omega)} \phi(i\omega).$$
For such transfer operators, it was proved (see \cite{gdms}) that if $\psi$ is 
H\"older continuous and summable, then there exists a unique fixed probability measure $\tilde\nu_\psi$ on $\Sigma_I^+$, such that $$\cL^*(\tilde\nu_\psi) = \tilde\nu_\psi.$$

The projection, denoted by $\tilde\mu_P$, of the measure $\tilde\nu_\psi$ onto the limit set $\Lambda$ of $\mathcal S$, is the stationary measure of the system $\mathcal S$ with the place-dependent probabilities $P$.
Hence, $$\tilde \mu_P = \pi_*\tilde\nu_\psi$$
On the other hand, for the expanding map $\sigma: \Sigma_I^+ \to \Sigma_I^+$ and the H\"older continuous potential $\psi$, let us notice that there exists also a shift-invariant equilibrium measure $\mu_\psi$ of $\psi$ on $\Sigma_I^+$, and moreover  there exists a function $\theta$ such that $$\theta\tilde\mu_\psi = \mu_\psi$$
Moreover there exists a constant $M$ depending on $\psi$, such that the above function $\theta$ satisfies $$\theta \ge M>0.$$ 
On the other hand, the projection of the shift-invariant equilibrium measure $\mu_\psi$ onto the limit set $\Lambda$ is denoted by $\mu_P$ and we  have $$\mu_P = \pi_* \mu_\psi.$$
We proved in \cite{MU-Adv} that the projection $\mu_P$ of the invariant measure $\mu_\psi$ is exact dimensional even when the system $\mathcal S$ has overlaps; and we found a formula for its pointwise dimension, involving the \textit{projection entropy} recalled above in (\ref{pent}). Indeed in the main Theorem of \cite{MU-Adv}, it is enough to consider a random countable iterated function system where the parameter space consists of only one point, and to take the identity as the evolution map on the space of  parameters.
This means that for $\mu_P$-a.e $x \in \Lambda$,
\begin{equation}\label{dimmu}
\mathop{\lim}\limits_{r\to 0} \frac{\log \mu_P(B(x, r))}{\log r} = \delta =  \frac{h_{\mu_\psi}(\mathcal S)}{\chi_{\mu_\psi}},
\end{equation}
and $\delta$ does not depend on $x \in \Lambda$.
But $\mu_P(B(x, r)) = \mu_\psi(\pi^{-1}(B(x, r)))$, and we know that $\mu_\psi = \theta \tilde \mu_\psi$, hence 
$$\tilde \mu_\psi(\pi^{-1}(B(x, r))) = \int_{\pi^{-1}(B(x, r))} \theta d\mu_\psi.$$
On the other hand let us recall that $\theta > M$ and that $\theta$ is a continuous bounded function on $\Sigma_I^+$. Hence using (\ref{dimmu}) and the fact that $\tilde \mu_P = \pi_* \tilde \mu_\psi$, we see that for $\mu_P$-a.e $x\in \Lambda$,
\begin{equation}\label{ptwsta}
\mathop{\lim}\limits_{r\to 0} \frac{\log \tilde \mu_P(B(x, r))}{\log r} = \delta.
\end{equation}

Therefore, the stationary measure $\tilde \mu_p$ of the system $\mathcal S$ with the place-dependent probabilities $P$ is exact dimensional, and from (\ref{ptwsta}) its Hausdorff (and pointwise) dimension is given by:
\begin{equation}\label{dimform}
HD(\tilde \mu_p) = \frac{h_{\mu_\psi}(\mathcal S)}{\chi_{\mu_\psi}},
\end{equation}
where $h_{\mu_\psi}(\mathcal S)$ is the projection entropy of  $\mu_\psi$ with respect to $\mathcal S$ and 
$ \chi_{\mu_\psi}$ is its Lyapunov exponent. From the definition (\ref{pent}) of the projection entropy $h_{\mu_\psi}(\mathcal S)$, it follows that $$h_{\mu_\psi}(\mathcal S) \le h_{\mu_\psi}(\sigma).$$ Therefore, we obtain the desired dimension formula.

\end{proof}

\section{Unfolding of countable IFS with overlaps and families of fractals.}\label{maximal}

We want now to associate  a \textit{Smale endomorphism} to a \textit{countable iterated function systems with overlaps $\mathcal S$}, by unfolding, in such a way to control the structure of overlappings. Then we will consider equilibrium measures of real-valued summable functions $\psi$ on $\Sigma_I$, and will study their projection measures on families of subfractals in the limit set $\Lambda$ of the  iterated function system $\mathcal S$.

Let us consider a countable conformal IFS with overlaps $\mathcal S = \{\phi_i, i \in I\}$, where the maps $\phi_i:V \to V$ are conformal contractions defined on a neighbourhood of a compact set $V \subset \mathbb R^D$, and $|\phi_i'| < \alpha < 1, i \in I$ on $V$. Denote by $\Lambda$ the limit set of $\mathcal S$, and assume that $\Lambda$ is not contained in the boundary of $V$. Since we work with a countable system, the limit set $\Lambda$ may be non-compact.

 Assume next that $I = \mathbb N^*$ and that our system $\mathcal S$ satisfies the Bounded Distortion Property, i.e
there exist constants  $H>0, \beta>0$ such that for all $i \in I$,
\begin{equation}\label{BDP}
|\log|\phi_i'(y)| - \log|\phi_i'(z)|| \le H|y-z|^\beta, \ \forall y, z \in V.
\end{equation} 

\

We will associate to the countable system with overlaps $\mathcal S$, a Smale skew product $T$ (and a random system with complete connections), with the goal  to separate the images of the compositions of maps $\phi_i$ along any given sequence $\omega = (\omega_0, \omega_1, \ldots) \in \Sigma_I^+$. This is realised by an inductive process of \textit{unfolding the overlaps} of $\cS$. 

Recall that $\Lambda$ is the limit set of $\mathcal S$.
Assume that, for any point  $x\in \Lambda$, the $\mathcal S$-images $\phi_i(x), i \in I$ of $x$,  satisfy the following \textbf{Non-accumulation Condition}:
\begin{equation}\label{nonrec}
\phi_i(x) \notin \ \overline{\{\phi_j(x), j \in I \setminus \{i\}\}}, \ \forall i \in I.
\end{equation}
This condition is quite general, and it can be checked on many systems (see for instance  examples in   \cite{MU-CM}).
 Recall that $\pi(\omega) = \phi_{\omega_0\omega_1\ldots}$ is the canonical projection from $\Sigma_I^+$ to the limit set $\Lambda$. Then, for an arbitrary $\omega = (\omega_0, \omega_1, \ldots) \in \Sigma_I^+$, we define inductively the contractions $T_{i\omega}$ for $i \in I$.
Let us start by defining $$T_{1\omega}:= \phi_{1\omega_0\ldots \omega_{n_1(\omega)}} = \phi_{1}\circ \phi_{\om_0}\circ \ldots \circ \phi_{\om_{n_1(\om)}},$$ where $n_1(\omega)$ is defined as the smallest integer $n_1\ge 1$ such that $$\phi_j(\pi(\omega)) \notin \phi_{1\omega_0\ldots \omega_{n_1}}(V),  \ \text{for all} \  j \ne 1$$ 
Next, since from above $\phi_2(\pi(\omega)) \notin \phi_{1\omega_0\ldots \omega_{n_1}}(V)$, take  $n_2(\omega)$ to be the smallest integer $n_2 >n_1$ such that $$\phi_{2\omega_0\ldots \omega_{n_2(\omega)}}(V) \cap \phi_{1\omega_0\ldots \omega_{n_1(\omega)}}(V) = \emptyset, \ \text{and} \ \ \phi_j(\pi\omega) \notin \phi_{2\omega_0\ldots \omega_{n_2(\omega)}}(V), \ \text{for} \ j \ne 2$$
Then we define $$T_{2\omega} := \phi_{2\omega_0\ldots \omega_{n_2(\omega)}}$$

Inductively, if we defined $n_k(\omega)$ up to some $k \ge 1$, we now define $n_{k+1}(\omega) \ge 1$ as the smallest integer $n_{k+1} > n_k$ with the property that  
\begin{equation}\label{indk}
\phi_{(k+1)\omega_0\ldots\omega_{n_{k+1}(\omega)}}(V) \cap \phi_{j\omega_0 \ldots \omega_{n_j(\omega)}}(V) = \emptyset, \ 1\le j \le k, \ \text{and} \ \ \phi_{\ell}(\pi\omega) \notin \phi_{(k+1)\omega_0\ldots\omega_{n_{k+1}(\omega)}}(V), \ell \ne k+1,
\end{equation}
and then the fiber map
\begin{equation}\label{Tk}
T_{k+1 \omega}:= \phi_{(k+1)\omega_0\ldots \omega_{n_{k+1}(\omega)}}
\end{equation}

Thus for any $\om \in \Sigma_I^+$, we construct as above a contraction map $T_\om: V \to V$,
\begin{equation}\label{contmapT}
T_\om = T_{\om_0\sigma(\om)}.
\end{equation}

\begin{dfn}\label{mSs}
Using (\ref{contmapT}) let us  define the space $\hat Y := \Sigma_I^+ \times V$, and the skew product $T: \hat Y \to \hat Y$, 
$$T(\omega, x) = (\sigma(\omega), T_\omega(x)), \ (\omega, x) \in \hat Y.$$
We will call $T: \Sigma_I^+\times V \to \Sigma_I^+ \times V$ the \textbf{maximal Smale system}  associated to the countable IFS with overlaps $\mathcal S$. 
\end{dfn}

From definition we see that the map $T_\omega$ depends on the whole sequence $\omega \in \Sigma_I^+$, and not just on the projection point $\pi(\omega)$. Thus the dynamical system $(\hat Y, T)$ describes \textit{how} the overlappings of $\mathcal S$  are formed through succesive iterations.

Consider now a sequence $\tau \in \Sigma_I$, and for any $n \ge 1$  define the composition map $$T_\tau^n = T_{\tau_{-1}^\infty}\circ T_{\tau|_{-2}^\infty} \circ \ldots \circ T_{\tau|_{-n}^\infty}.$$
Then $\hat\pi_2(\tau) = \mathop{\cap}\limits_{n=1}^\infty \overline{T_\tau^n(V)}$, and the fractal $J_\omega = \hat\pi_2([\omega])$ is contained in $\Lambda$, where $[\omega]$ is the cylinder in $\Sigma_I$ determined by $\omega \in \Sigma_I^+$. 

\

If $\mu$ is a probability measure on $\Sigma_I$, define its \textit{Lyapunov exponent with respect to $T$} as 
$$\chi_{\mu}(T) = - \int_{\Sigma_I} \log |T'_{\tau|_0^\infty}(\hat\pi_2(\tau))| \ d\mu(\tau).$$ 
Denote by $\pi_0:\Sigma_I \to \Sigma_I^+$, $\pi_0(\tau) = (\tau_0, \tau_1, \ldots)$, the canonical truncation map.

\

We constructed above the family of random fractals $J_\omega\subset \Lambda, \omega \in \Sigma_I^+$, and now we prove that certain  probability measures are exact dimensional on $J_\om$.

\begin{thm}\label{exactmax}
Let a countable IFS with overlaps $\mathcal S$ as above satisfying (\ref{nonrec}), and let  $T$ be its associated maximal Smale system. Consider $\psi:\Sigma_I \to \mathbb R$ summable H\"older continuous function with equilibrium measure $\mu_\psi$,  and  let $\nu_\psi$ be the canonical projection $\pi_{0, *}\mu_\psi$ of $\mu_\psi$ on $\Sigma_I^+$. Take the conditional measure $\mu^\omega_\psi$ of $\mu_\psi$ on the cylinder $[\omega]$ and let $\nu_\psi^\omega := \hat\pi_{2 *} \mu^\omega_\psi$ be its projection on $J_\omega$, with $\hat\pi_{2}$ defined in \ref{5111705p141}.
Then, for $\nu_\psi$-a.e $\omega \in \Sigma_I^+$, the measure $\nu^\omega_\psi$ is exact dimensional on the subfractal $J_\om \subset \Lambda$, and 
$$
HD(\nu^\omega_\psi) = \frac{h_{\mu_\psi}(\sigma)}{\chi_{\mu_\psi}(T)}.$$
\end{thm}

\begin{proof}

First let us  notice that, from our construction, for any  $i \ne j$, $$T_{i\omega}(V) \cap T_{j\omega}(V) = \emptyset. $$ Therefore the open set condition in fibers from the definition of the Smale skew-product is satisfied. Moreover if $||\phi_i'|| < \alpha < 1, i \in I$, it follows that the same uniform contractivity condition is satisfied by all the maps $T_\omega, \omega \in \Sigma_I^+$. Hence the uniform contractivity of the maps $T_\omega$ is satisfied too.

Let us see now if the maps $T_\omega$ satisfy Bounded Distortion Property (BDP).
For this consider an arbitrary $\omega \in \Sigma_I^+$. Then $T_\omega = \phi_{\omega_0\omega_1\ldots \omega_n}$ for some integer $n$ which depends on $\omega$. 
Thus, there exists $L \in (0, 1)$ and $H'>0$ such that for any $y, z \in V$, we have
$$
\begin{aligned}
&|\log|T_{\omega}'(y)| - \log|T_\omega'(z)||  \le \\
&\le |\log |\phi_{\omega_0}(\phi_{\omega_1\ldots\omega_n}(y))| - \log|\phi_{\omega_0}(\phi_{\omega_1\ldots\omega_n}(z))| | + \ldots + |\log|\phi_{\omega_n}'(y)| - \log |\phi_{\omega_n}'(z)|| \le \\
& \le H|y-z|^\beta (1+L + \ldots + L^n) = H' |y-z|^\beta
\end{aligned}
$$

Another condition in the definition of a conformal Smale skew product is the H\"older continuity of the real-valued map on $\Sigma_I$ given by:
\begin{equation}\label{holderlog}
\tau \longrightarrow \log|T'_\tau(\hat \pi_2(\tau))|
\end{equation}

Let us take $\omega \in \Sigma_I^+$ and $\tau \in \Sigma_I$ such that $\tau \in [\omega]$. Then $T_\tau = T_{\tau_{-1}\omega}$, and consider the integer $n_{\tau_{-1}}(\omega)$. Then from definition, we have $$T_{\tau_{-1}\omega} = \phi_{\tau_{-1}}\circ \phi_{\omega_0}\circ \ldots \circ \phi_{\omega_{n_{\tau_{-1}}(\omega)}}$$

But if $\eta \in [\tau_{-m}\ldots \tau_m]$ for $m \ge n_{\tau_{-1}(\omega)}$, we have that $T_\tau = T_\eta$. On the other hand, due to the form of $\hat\pi_2(\tau)$ given in (\ref{pi2hat}), we have $$d(\hat\pi_2(\tau), \hat \pi_2(\eta)) \le C\frac{1}{2^m},$$
for some constant $C$ independent of $\tau, \eta, m$.
Hence the map from (\ref{holderlog}) is indeed H\"older continuous on $\Sigma_I$. 

Therefore, the maximal Smale system $T: \hat Y \to \hat Y$ defined above,  where we let $$T_{\tau|_0^\infty} = \phi_{\tau_0 \tau_1\ldots \tau_{n(\sigma\tau|_0^\infty)}},$$ satisfies the properties of a conformal Smale skew product endomorphism. We then apply our result from \cite{MU-ETDS}, to  obtain the exact dimensionality and the formula for the Hausdorff dimension of the projection measures $\nu^\omega_\psi$ on $J_\omega \subset \Lambda$. In particular, it follows that for $\nu_\psi$-a.e $\om \in \Sigma_I^+$,
$$
HD(\nu^\omega_\psi) = \frac{h_{\mu_\psi}(\sigma)}{\chi_{\mu_\psi}(T)},$$
where $h_{\mu_\psi}(\sigma)$ is the entropy of the measure $\mu_\psi$, and $\chi_{\mu_\psi}(T)$ is its Lyapunov exponent.

\end{proof}

Examples of functions $\psi$, based on the maximal Smale system $T$ associated to a countable IFS with overlaps  $\mathcal S$ as above, are given next. In Theorem \ref{geom} we also construct a family of measures on subfractals in the limit set $\Lambda$, which are related to the intricate geometry of  $S$. 

\begin{thm}\label{geom}
In the setting of Theorem \ref{exactmax} let the countable IFS $\mathcal S$ with limit set $\Lambda$, and for any $s>0$ define the function $\psi_s: \Sigma_I \to \R$,
 $$\psi_s(\eta) = s\log|T_{\eta|_0^\infty}'(\hat\pi_2(\eta))|, \ \eta \in \Sigma_I.$$  
Then for any $s>0$,

 a) 
  $\psi_s$ is summable and H\"older continuous on $\Sigma_I$. 
 
 b) If $\mu_{s}:= \mu_{\psi_s}$ is the equilibrium measure of $\psi_s$ on $\Sigma_I$, and $\nu_s:= \pi_{0*}\mu_s$ on $\Sigma_I^+$, then for $\nu_s$-a.e. $\om \in \Sigma_I^+$, the measure $\nu_s^\om:= \nu_{\psi_s}^\om$ is exact dimensional on  $J_\om \subset \Lambda$ and, $$HD(\nu_s^\om) = \frac{h_{\mu_{s}}(\sigma)}{ \big|\int_{\Sigma_I}\log |\phi_{\tau_0 \tau_1\ldots \tau_{n_{\tau_0}(\sigma\tau)}}'(\hat\pi_2(\tau))| \ d\mu_{s}(\tau)\big|}.$$ 
 \end{thm}
 
 \begin{proof}
 a) Since the maximal system $T$ associated to $\mathcal S$ by Definition \ref{mSs} was shown in Theorem \ref{exactmax} to be a conformal Smale endomorphism, it follows that $\psi_s$ is H\"older continuous on $\Sigma_I$.

Also  for $i \in I, \om \in \Sigma_I^+$, recall the definition of $T_{i\om}$ as a composition of maps given by the inductive relation (\ref{indk}). Also,  the initial contractions $\phi_i: V \to V$ from $\mathcal S$ are defined on a compact set $V\subset \R^D$, such that $|\phi_i'| \le \alpha < 1, i \in I$ on $V$.  
From the definition (\ref{indk}), we have the increasing sequence of integers, $n_1(\om) < n_2(\om) < \ldots,$ and thus for any $\om \in \Sigma_I^+$, the positive integers $n_k(\om)$ satisfy the inequality,
$$n_k(\om) \ge k, \ k \ge 1.$$
Now we have from (\ref{indk}) that $T_{i\om} = \phi_{i\om_0\ldots \om_{n_i(\om)}}$, and thus it follows from above that for any $i \in I$, $$|T_{i\om}'| \le \alpha ^{n_i(\om)} \le \alpha^{i}.$$  Therefore, since we assumed  $s>0$, it follows that, $$\mathop{\sum}\limits_{i \in I} e^{\sup\psi_s|_{[i]}} \le \mathop{\sum}\limits_{k\ge 1} \alpha^{sk}  < \infty. $$ Hence the real-valued function $\psi_s$ is also summable in this case. 
  
  b) For any $s >0$, if $\nu_s:= \nu_{\psi_s}$ is the canonical projection of the equilibrium measure $\mu_{\psi_s}$ onto $\Sigma_I^+$, then in the notation of Theorem \ref{exactmax} we have for $\nu_s$-a.e $\om \in \Sigma_I^+$,  $$\nu_s^\om := \nu_{\psi_s}^\om = \hat\pi_{2*}\mu_{\psi_s}^\om,$$ which is a measure supported on $J_\om \subset \Lambda$. Now due to the properties of $\psi_s$ proved in a), one can apply Theorem \ref{exactmax} for the measure $\nu_s^\om$ on the fiber set $J_\om \subset \Lambda$, for $\nu_s$-a.e. $\om \in \Sigma_I^+$. 
  Thus we obtain the exact dimensionality of $\nu_s^\om$ on $J_\om$, and that $$HD(\nu_s^\om) = \frac{h_{\mu_s}(\sigma)}{\chi_{\mu_s}(T)}.$$ Then, by using the expression of $T_{i\om}$ to compute $\chi_{\mu_s}(T)$ in our case, we obtain the Hausdorff dimension of $\nu_s^\om$ by the above formula.  
  
  \end{proof}
  
  \


\begin{thebibliography}{99}


\bibitem{BDEG} M. F. Barnsley, S. Demko, J. Elton, J. Geronimo, Invariant measures for Markov processes arising from iterated function systems with place-dependent probabilities, Ann Inst H Poincar\'e, 24, 1988, 367-394.

\bibitem{BD}
M. F. Barnsley, S. Demko, Iterated function systems and the global construction of fractals, Proc Royal Soc London A, 399 (1817): 243–275, 1985.

%\bibitem{BV1}
%M. F. Barnsley, A. Vince, Fast basins and branched fractal manifolds of attractors of iterated function systems, Symmetry, Integrability and Geometry, 11, 2015, 084.

\bibitem{BV}
M. F. Barnsley, A. Vince, The chaos game on a general iterated function system, Ergodic Th Dynam Syst 31 (2011), 1073-1079.

%\bibitem{Bo}
%R. Bowen, Equilibrium States and Ergodic Theory for Anosov Diffeomorphisms, Springer, 1975.

%\bibitem{DLN}
%Q-R. Deng, K-S. Lau, S-M. Ngai, Separation conditions for iterated function systems with overlaps, Contemp Math, Amer Math Soc, vol 600, 2013, 1-21.

\bibitem{DF}
W. Doeblin, R. Fortet, Sur des chaines a liaisons compl\'etes, Bull Soc Math France, 65, 1937, 132-148.

\bibitem{Fa}
K. Falconer, Techniques in Fractal Geometry, J. Wiley \& Sons, Chichester, 1997.

\bibitem{FL}
A H. Fan,  K-S. Lau, Iterated function system and Ruelle operator, J. Math. Anal. Appl. 231 (1999), no. 2, 319–344.

\bibitem{FH}
D. J. Feng, H. Hu,
Dimension theory of iterated function systems, Commun. Pure
and Applied Math., 62 (2009), 1435-1500.

%\bibitem{FM}
%J.E. Fornaess,  E. Mihailescu, Equilibrium measures on saddle sets of holomorphic maps on $\mathbb P^2$, Math. Ann. 356 (2013), 1471-1491.

\bibitem{GP}
S. Grigorescu, G. Popescu, Random systems with complete connections as a framework for fractals, Stud Cerc Math, vol 41, 1989, 481-489.

%\bibitem{HLW}
%T-Y. Hu, K-S. Lau, X-Y. Wang, On the absolute continuity of a class of invariant measures, Proc. Amer. Math. Soc. 130 (2002), 759-767.

\bibitem{H}
J. E. Hutchinson, Fractals and self-similarity, Indiana Univ. Math. J. 30 (1981) 713-747.

\bibitem{ITM} C. Ionescu Tulcea, G. Marinescu, Sur certaines chaines a liaisons compl\'etes, C. R. Acad Sci Paris, 227, 1948, 667-669.

%\bibitem{I}
%M. Iosifescu, Iterated function systems. A critical survey. Math. Rep. 11 (61) (2009), no. 3, 181--229.

\bibitem{IG} M. Iosifescu, S. Grigorescu, Dependence with complete connections and its applications, Cambridge Univ Press, 1990.


\bibitem{L}
D. La Torre, E. Maki, F. Mendivil, E.R Vrscay,
Iterated function systems with place-dependent probabilities and the inverse problem of measure approximation using moments, 
Fractals 26 (2018), no. 5, 1850076.

\bibitem{LP}
F. Ladjimi, M. Peign\'e, On the asymptotic behavior of the Diaconis–Freedman chain on $[0, 1]$, Statist. Probab. Lett. 145 (2019), 1-11.

\bibitem{LNW}
K-S. Lau, S-M. Ngai, X-Y. Wang, Separation conditions for conformal iterated function systems, Monatsh. Math. 156 (2009), no. 4, 325-355.

\bibitem{gdms}
D. Mauldin, M. Urba\'nski, 
Graph Directed Markov
Systems: Geometry and Dynamics of Limit Sets, Cambridge University Press,
2003.

\bibitem{M-CCM}
E. Mihailescu, Thermodynamic formalism for invariant measures
in iterated function systems with overlaps, Commun. Contemp. Math, online 2021, DOI: 10.1142/S0219199721500413.

%\bibitem{M-Monat}
%E. Mihailescu, On some coding and mixing properties on folded fractals, Monatsh Math 167, 2012, 241-255.

\bibitem{M-stable} E. Mihailescu, On a class of stable conditional measures, Ergodic Th Dynam Syst, 31, 2011, 1499-1515.

%\bibitem{M-MZ}
%E. Mihailescu, Unstable directions and fractal dimension for skew products with overlaps in fibers, Math Zeitschrift, 269, 2011, 733-750.


\bibitem{M-DCDS}
E. Mihailescu, Equilibrium measures, prehistories distributions and fractal dimensions for endomorphisms, Discrete Cont Dynam Syst, 32, 2012, 2485-2502.

\bibitem{MS}
E. Mihailescu, B. Stratmann, Upper estimates for stable dimensions on fractal sets with variable numbers of foldings, Int Math Res Notices 2014 (23), 6474-6496.

\bibitem{MU-Adv}
E. Mihailescu, M. Urbanski, Random countable iterated function systems with overlaps and applications, Advances in Math, 298, 2016, 726-758.


\bibitem{MU-ETDS}
E. Mihailescu, M. Urba\'nski, Skew product Smale endomorphisms over countable shifts of finite type, Ergodic Th Dynam Syst, 40, 3105-3149, 2020.

\bibitem{MU-CM} E. Mihailescu, M. Urba\'nski, Hausdorff dimension of limit sets of countable conformal iterated function systems with overlaps, Contemp. Math., vol 600, American Math Soc, 2013, 273–290.

\bibitem{NT}
S-M. Ngai, J-X. Tong, Infinite iterated function systems with overlaps, Ergodic Th Dynam Syst (2016), 36, 890-907.

\bibitem{OM2}
O. Onicescu, G. Mihoc, Sur les chaines de variables statistiques, Bull Sciences Math, 59, 1935, 174-192.

\bibitem{PS}
Y. Peres, B. Solomyak, Existence of $L^q$ dimensions and entropy dimension for
self-conformal measures, Indiana Univ. Math. J. 49 (2000), 1603-1621.

\bibitem{Pe}
Y. Pesin, Dimension Theory in Dynamical Systems, Chicago Lectures in Mathematics, 1997.


\bibitem{Ro}
V. A Rokhlin, Lectures on the theory of entropy of transformations with invariant measures, Russian Math Surveys, 22, 1967, 1-54.

\bibitem{Ru}
D. Ruelle, Thermodynamic formalism. The mathematical structures of equilibrium statistical mechanics, Second edition, Cambridge University Press, Cambridge, 2004.

\bibitem{S}
\"O. Stenflo, 
Uniqueness of invariant measures for place-dependent random iterations of functions,  IMA Vol. Math. Appl., 132, 13--32, Springer, New York, 2002.

\bibitem{Y}
L. S. Young, 
Dimension, entropy and Lyapunov exponents, Ergodic Th  Dynam Syst, vol 2, 1982, 109-124.

\

\end{thebibliography}
\end{document}